\documentclass[amsthm,leqno]{monsky2009} % for arXiv.

\usepackage{graphicx}
\usepackage{amssymb}
\usepackage{amsmath}

\usepackage{bm}

\newtheorem*{corollary*}{Corollary}
\newtheorem*{remark*}{Remark}
\newtheorem*{remarks*}{Remarks}
\newtheorem*{proofidea*}{Idea of proof}
\newtheorem*{speculation*}{Speculation}

\newtheorem{definition}{Definition}[section] %section number prepended.
\newtheorem{theorem}[definition]{Theorem}
\newtheorem{lemma}[definition]{Lemma}

\newtheorem{example}[]{Example} % simple numbering
\newtheorem{corollary}[definition]{Corollary}
\newtheorem{remark}[]{Remark} % simple numbering

\newcommand{\conste}{\ensuremath{\mathrm{e}}}
\newcommand{\consti}{\ensuremath{\mathrm{i}}}

\newcommand{\newz}{\ensuremath{\mathbb{Z}}}

\begin{document}

\begin{frontmatter}



%\title{}


%\author{}

\title{The reciprocals of some characteristic 2 ``theta series''}
\author{Paul Monsky}

\address{Brandeis University, Waltham MA  02454-9110, USA\\  monsky@brandeis.edu }

\begin{abstract}
Suppose $l=2m+1$, $m>0$. We introduce $m$ ``theta-series'', $[1],\ldots ,[m]$, in $\newz/2[[x]]$. It has been conjectured that the $n$ for which the coefficient of $x^{n}$ in $1/[i]$ is 1 form a set of density 0. This is probably always false, but in certain cases, for $n$ restricted to certain arithmetic progressions, it is true. We prove such zero-density results using the theory of modular forms, and speculate about what may be true in general.
\end{abstract}
\maketitle
%\begin{keyword}

%\PACS 
%\end{keyword}
\end{frontmatter}

% main text

\section{Introduction}
\label{section1}

Throughout $L$ is a field of fractions of $\newz/2[[x]]$, viewed as the field of Laurent series with coefficients in $\newz/2$.

\begin{definition}
\label{def1.1}
For $g \ne 0$ in $\newz/2[[x]]$, $B(g)$ is the set of $n$ in $\newz$ for which the co-efficient of $x^{n}$ in $1/g$ is 1. Note that only finitely many elements of $B(g)$ can be $<0$.
\end{definition}

Fix $l=2m+1$ with $m>0$. We define certain ``theta series'' $[i]$ in $\newz/2[[x]]$.

\begin{definition}
\label{def1.2}
$[i]=\sum x^{n^{2}}$, the sum extending over all $n$ in $\newz$ with $n\equiv i\pod{l}$. (Note that $[0]=1$, and that $[i]=[j]$ whenever $i\equiv \pm j\pod{l}$. So the ring $S$ generated over $\newz/2$ by all the $[i]$ is just $\newz/2[[1],\ldots ,[m]]$.)
\end{definition}

In this note we study the sets $B([r])$ for fixed $l$ and $r$ with $r$ prime to $l$. Note that each $j$ in $B([r])$ is $\equiv -r^{2}\pod{l}$ and that consequently $B([r])$ has (upper) density at most $1/l$ in the positive integers.

In \cite{1}, Cooper, Eichhorn and O'Bryant conjectured, in a slightly different language, that each $B([r])$ has density 0. I think this is never true, but we'll show that for certain $l$ and $r$ and in certain congruence classes $\mod$ a power of 2, $B([r])$ indeed has relative density 0. For example when $l=3$ the relative density is 0 in the classes $n\equiv 0\pod{2}$, $n\equiv 1\pod{4}$ and $n\equiv 3\pod{8}$. I'll now describe more precisely, what perhaps is true in general, and the small part of it I'm able to prove.

\begin{definition}
\label{def1.3}
Fix $l$. $k<0$ is ``$l$-exceptional'' if $k$ is in some $B([r])$ with $r$ prime to $l$. A ``basic congruence class'' is a congruence class of the form $n\equiv k\pod{8q}$, where $k$ is $l$-exceptional and $q$ is the largest power of 2 dividing $k$.
\end{definition}

\begin{definition}
\label{def1.4}
An integer $n\ge 0$ is in $U$ if it is in some basic congruence class, and in $U^{*}$ otherwise.
\end{definition}

\begin{example}
\label{example1}
Suppose $l=3$. Then $1/[1]=x^{-1}+\cdots$. So the only $3$-exceptional $k$ is $-1$ and the only basic class is $n\equiv -1\pod{8}$. $U^{*}$ consists of the integers $n\ge 0$ with $n\equiv 0\pod{2}$, $n\equiv 1\pod{4}$, or $n\equiv 3\pod{8}$.
\end{example}

\begin{example}
\label{example2}
Suppose $l=9$. The only $[r]$ we need consider are $[1]$, $[2]$ and $[4]$. Now $1/[1]=x^{-1}+\cdots$, $1/[2]=x^{-4}+\cdots$ and $1/[4]=x^{-16}+x^{-7}+\cdots$. So the basic classes are $n\equiv 1$ or $-1\pod{8}$, $n\equiv -4\pod{32}$ and $n\equiv -16\pod{128}$. Then $U$ consists of the integers $\ge 0$ lying in $16+16+4+1=37$ congruence classes to the modulus 128, and $U^{*}$ of the integers $\ge 0$ in the remaining 91 classes.
\end{example}

It seems to me plausible that when $r$ is prime to $l$ then $B([r])$ has relative density 0 in $U^{*}$. I'll show that this holds for $l\le 11$. When $l=13$ or $15$, then $U^{*}$ is the union of $83 \mod 128$ congruence classes, and I'll prove that $B([r])$ has relative density 0 in each of these classes, with the possible exception of the class $n\equiv 48\pod{128}$.  Unfortunately the proof is not unified---we have to write $U^{*}$ as a union of congruence classes and examine each class in turn. To this end we now give the (easily proved) description of $U^{*}$ as a union of congruence classes for each $l\le 15$.

$$\begin{array}{r@{\hspace{.5em}}r@{\hspace{.5em}}r@{\hspace{.5em}}r@{\hspace{.5em}}r@{\hspace{.5em}}r@{\hspace{.5em}}r@{\hspace{.5em}}r}
l & \mod{2} & \mod{4} & \mod{8} & \mod{16} & \mod{32} & \mod{64} & \mod{128}\\\hline
3 &  0 &  1 &  3\\
5 & & 1,2 &  0,3 &  4 &  12\\
7 & & 1 &  0,2,3 &  4,6 &  12\\
9 & & 2 &  3,5 &  4,8 &  0,12 &  16 &  48\\
11 & && 1,3,6 &  4,8,10 &  0,12 & 16 &  48\\
13 & && 2,3,5 &  4,8,14 &  0,12 & 16 &  48\\
15 & && 1,2,3 &  4,6,8 &  0,12 & 16 &  48\\
\end{array}
$$

%$$\begin{array}{rclllllll}
%l=3 &\qquad& n\equiv 0\pod{2},\quad n\equiv 1\pod{4},\quad n\equiv 3\pod{8}\\
%l=5 &\qquad& n\equiv 1,2\pod{4},\quad n\equiv 0,3\pod{8},\quad n\equiv 4\pod{16},\quad n\equiv 12\pod{32}\\
%l=7 &\qquad& n\equiv 1\pod{4},\quad n\equiv 0,2,3\pod{8},\quad n\equiv 4,6\pod{16},\quad n\equiv 12\pod{32}\\
%l=9 &\qquad& n\equiv 2\pod{4},\quad n\equiv 3,5\pod{8},\quad n\equiv 4,8\pod{16},\\&\quad&  n\equiv 0,12\pod{32},\quad n\equiv 16\pod{64},\quad n\equiv 48\pod{128}\\
%l=11 &\qquad& n\equiv 1,3,6\pod{8},\quad n\equiv 4,8,10\pod{16},\quad n\equiv 0,12\pod{32}, \\&\quad& n\equiv 16\pod{64},\quad n\equiv 48\pod{128}\\
%l=13 &\qquad& n\equiv 2,3,5\pod{8},\quad n\equiv 4,8,14\pod{16},\quad n\equiv 0,12\pod{32}, \\&\quad&  n\equiv 16\pod{64},\quad n\equiv 48\pod{128}\\
%l=15 &\qquad& n\equiv 1,2,3\pod{8},\quad n\equiv 4,6,8\pod{16},\quad n\equiv 0,12\pod{32}, \\&\quad&  n\equiv 16\pod{64},\quad n\equiv 48\pod{128}\\
%\end{array}
%$$

Here's a rough description of how our proofs proceed. Fix $l$ and $[r]$ and a congruence class $j \mod q$ where $q$ is a power of 2. We'll construct a $g$ in  $\newz/2[[x]]$, depending on $l$, $r$, $j$ and $q$, with the following properties:

\begin{enumerate}
\addtolength{\itemsep}{1ex}
\item[(1)] There are integers $c_{0}, c_{1}, \ldots$ such that: \\[-1.8ex]
\begin{enumerate}
\addtolength{\itemsep}{1ex}

\item[(A)] $\sum c_{n}\conste^{2\pi\consti nz}$ converges in $\operatorname{Im} (z)>0$ to a modular form of integral weight for a congruence group.
\item[(B)] $g$ is the mod 2 reduction of $\sum c_{n}x^{n}$ 
\end{enumerate}

\item[(2)] Suppose that $g/[r]^{q}$ is itself the mod 2 reduction of some $\sum d_{n}x^{n}$ where $\sum d_{n}\conste^{2\pi\consti nz}$ converges to a modular form as in 1(A) above. Then $B([r])$ has density 0 in the congruence class $j \mod q$.
\end{enumerate}

$g$ is in fact the image of $[r]^{q-1}$ under a certain projection operator $p_{q,j}$ which we describe in the next section. The fact that $g$ is ``the reduction of a modular form'' comes from a corresponding result for $[r]$; $[r]$ is the reduction of a weight 1 modular form. (The proof of (2) is deeper, coming from a result of Deligne and Serre on the reduction of modular forms.) Once (1) and (2) are established we still need to show that for each of our choices of $l$, $[r]$, and the congruence class $j \mod q$ lying in $U^{*}$, the power series $g/[r]^{q}$ satisfies the condition (2) above. This is true, for example, whenever $g/[r]^{q}$ lies in the ring $S$ of Definition \ref{def1.2}. In certain cases, extensive computer calculations tell us that $g/[r]^{q}$ lies in $S$.

At the end of the paper we'll speculate on the relative density of $B([r])$ in the basic classes. Though we are unable to prove anything, computer calculations suggest that each $B([r])$ has relative density $1/(2l)$ in each basic class.

\section{The operators $\bm{p_{q,j}}$ and the case $\bm{l=3}$}
\label{section2}

If $q$ is a power of 2, let $L^{[q]}\subset L$ consist of all $q$th powers of elements of $L$. $L$ is the direct sum of the $L^{[q]}$ vector-spaces $x^{j}L^{[q]}, 0\le j<q$.

\begin{definition}
\label{def2.1}
$p_{q,j}L\rightarrow x^{j}L^{[q]}$ is the $L^{[q]}$-linear projection map attached to the above direct sum decomposition.
\end{definition}

Note that $p_{q,j}(FG)=\sum p_{q,a}(F)p_{q,b}(G)$, the sum extending over all pairs $(a,b)$ with $a+b\equiv j\pod{q}$. Furthermore $p_{2q,2j}\left(F^{2}\right)=\left(p_{q,j}\left(F\right)\right)^{2}$. We'll use these facts often.

\begin{lemma}
\label{lemma2.2}
Fix $l=2m+1$. Then:

\begin{enumerate}
\item[(1)] $p_{2,0}([2i])=[i]^{4}$
\item[(2)] The subring $S$ of $L$ generated over $\newz/2$ by all the $[i]$ is stabilized by the operators $p_{8,0},\ldots ,p_{8,7}$.
\end{enumerate}
\end{lemma}

\begin{proof}
Since $[2i]=\sum_{n\equiv 2i\pod{l}}x^{n^{2}}$, $p_{2,0}([2i])=\sum_{k\equiv i\pod{l}}x^{4k^{2}}=[i]^{4}$.

In view of the formula for $p_{8,j}(FG)$, to prove (2) it suffices to show that $p_{8,0}([i]),\ldots ,p_{8,7}([i])$ are all in the subring. Now if $j\ne 0$, $1$ or $4$, each $p_{8,j}([i])$ is 0. Since every odd square is $\equiv 1\pod{8}$, $p_{8,1}([2i])=p_{2,1}([2i])=[2i]+[i]^{4}$. Also $p_{8,0}([4i])=p_{8,0}p_{2,0}([4i])=p_{8,0}\left([2i]^{4}\right)=\left(p_{2,0}([2i])\right)^{4}=[i]^{16}$. Similarly, $p_{8,4}([4i])=\left(p_{2,1}([2i])\right)^{4}=[2i]^{4}+[i]^{16}$.
\end{proof}

Suppose for the rest of this section that $l=3$. In this case the proofs of zero-density in $U^{*}$ are much easier than the proofs for $l>3$, requiring neither modular forms nor computer calculations. Observe that if 3 doesn't divide $i$, then $[i]=1$.

\begin{definition}
\label{def2.3}
$a=[1]=[2]$. Note that $p_{2,0}(a)=a^{4}$.
\end{definition}

\begin{theorem}
\label{theorem2.4}
Suppose $n\equiv 0\pod{2}$ and $n$ is in $B(a)$. Then $n/2$ is a square.
\end{theorem}

\begin{proof}
$p_{2,0}\left(\frac{1}{a}\right)=\frac{1}{a^{2}}p_{2,0}(a)=a^{2}$. Since $n$ is in $B(a)$ and is even, the coefficient of $x^{n}$ in $a^{2}$ is 1, giving the result.
\end{proof}

\begin{theorem}
\label{theorem2.5}
Suppose $n\equiv 1\pod{4}$ and $n$ is in $B(a)$. Then the number of pairs $(s_{1},s_{2})$ with $s_{1}$ and $s_{2}$ squares, and $s_{1}+4s_{2}=n$ is odd. Furthermore $n$ is the product of a prime and a square.
\end{theorem}

\begin{proof}
$p_{4,1}\left(\frac{1}{a}\right)=\frac{1}{a^{4}}p_{4,1}(a^{3})=\frac{1}{a^{4}}p_{4,1}(a)p_{4,1}(a^{2})=\frac{1}{a^{4}}\left(a+a^{4}\right)a^{8}=a^{5}+a^{8}$. Since $n$ is in $B(a)$ and is $\equiv 1\pod{4}$, the coefficient of $x^{n}$ in $a^{5}+a^{8}$ is 1, and so the coefficient in $a^{5}=a\cdot a^{4}$ is 1. So the number of pairs $(r_{1},r_{2})$ with $r_{1}\equiv r_{2}\equiv 1\pod{3}$ and $r_{1}^{2}+4r_{2}^{2}=n$ is odd. To each such pair attach the pair $\left(s_{1},s_{2}\right)$ with $s_{1}$ and $s_{2}$ squares, $s_{1}+4s_{2}=n$, by setting $s_{i}=r_{i}^{2}$. The function from pairs $(r_{1},r_{2})$ to pairs $(s_{1},s_{2})$ is 1--1. Since $n$ is in $B(a)$, $n\equiv -1\pod{3}$. So whenever we have a pair $(s_{1},s_{2})$ as above, $s_{1}$ and $s_{2}$ are $\equiv 1\pod{3}$ and have square roots $\equiv 1\pod{3}$.  So the function $(r_{1},r_{2})\rightarrow (s_{1},s_{2})$ is onto, and we get the first assertion of the theorem.  A little arithmetic in $\newz[i]$ gives the second assertion.
\end{proof}

\begin{lemma}
\label{lemma2.6}
If $n\equiv 3\pod{8}$, $n$ is in $B(a)$ if and only if the number of triples $(r_{1},r_{2},r_{3})$ with $r_{1}\equiv r_{2} \equiv r_{3}\equiv 1\pod {3}$ and $r_{1}^{2}+2r_{2}^{2}+8r_{3}^{2}=n$ is odd.
\end{lemma}

\begin{proof}
$p_{8,3}\left(\frac{1}{a}\right)=\frac{1}{a^{8}}p_{8,3}\left(a\cdot a^{2}\cdot a^{4}\right)=\frac{1}{a^{8}}p_{8,1}(a)p_{8,2}\left(a^{2}\right)p_{8,0}\left(a^{4}\right)=\frac{1}{a^{8}}\left(a+a^{4}\right) \linebreak \left(a+a^{4}\right)^{2}a^{16}=a^{11}+a^{14}+a^{17}+a^{20}$.  Since $n\equiv 3\pod{8}$ the coefficients of $x^{n}$ in $a^{14}$, $a^{20}$, and $a^{17}=a\cdot a^{16}$ are evidently 0. So $n$ is in $B(a)$ if and only if the coefficient of $x^{n}$ in $a^{11}=a\cdot a^{2}\cdot a^{8}$ is 1, giving the lemma.
\end{proof}

\begin{lemma}
\label{lemma2.7}
If $n\equiv 11\pod{24}$ the number of triples $(s_{1},s_{2},s_{3})$ where the $s_{i}$ are squares and $s_{1}+s_{2}+s_{3}=n$ is $3\cdot $(the number of triples $(r_{1},r_{2},r_{3})$ as in Lemma \ref{lemma2.6}).
\end{lemma}

\begin{proof}
If the $s_{i}$ are as above, two of them are $\equiv 1\pod{3}$ while 3 divides the third. So our lemma states that the number of triples $(s_{1},s_{2},s_{3})$ with the $s_{i}$ squares, $s_{1}+s_{2}+s_{3}=n$ and $s_{3}\equiv 0\pod{3}$ is the number of triples $(r_{1},r_{2},r_{3})$ as in Lemma \ref{lemma2.6}.  If we have a triple $(r_{1},r_{2},r_{3})$ let $s_{1}=r_{1}^{2}$, $s_{2}=(r_{2}-2r_{3})^{2}$, $s_{3}=(r_{2}+2r_{3})^{2}$. Then the $s_{i}$ are squares, $s_{3}\equiv 0\pod{3}$ and $s_{1}+s_{2}+s_{3}=r_{1}^{2}+2r_{2}^{2}+8r_{3}^{2}=n$. That $(r_{1},r_{2},r_{3})\rightarrow (s_{1},s_{2},s_{3})$ is 1--1 is easily seen.  To prove ontoness suppose we're given $(s_{1},s_{2},s_{3})$. Then $s_{1}$ and $s_{2}$ are $\equiv 1\pod{3}$ and have square roots, $\sqrt{s_{1}}$ and $\sqrt{s_{2}}$, that are $\equiv 1\pod{3}$. Also, since $n\equiv 3\pod{8}$, the $s_{i}$ are odd. So we can find a square-root, $\sqrt{s_{3}}$ of $s_{3}$ with $\sqrt{s_{3}}\equiv\sqrt{s_{2}}\pod{4}$. Then the triple $\left(\sqrt{s_{1}},\frac{-\sqrt{s_{2}}-\sqrt{s_{3}}}{2},\frac{\sqrt{s_{2}}-\sqrt{s_{3}}}{4}\right)$ has its entries $\equiv 1\pod{3}$ and maps to $(s_{1},s_{2}, s_{3})$.
\end{proof}

\begin{theorem}
\label{theorem2.8}
Suppose $n\equiv 3\pod{8}$ and $n$ is in $B(a)$. Then the number of pairs $(s_{1},s_{2})$ with $s_{1}$ and $s_{2}$ squares and $s_{1}+2s_{2}=n$ is odd. Furthermore, $n$ is the product of a prime and a square.
\end{theorem}

\begin{proof}
Consider the set of triples $(s_{1},s_{2},s_{3})$ where the $s_{i}$ are squares and $s_{1}+s_{2}+s_{3}=n$. Since $n$ is in $B(a)$, and $n\equiv 3\pod{8}$, $n\equiv 11\pod{24}$. Lemmas \ref{lemma2.6} and \ref{lemma2.7} then show that the number of such triples is odd. Now $(s_{1},s_{2},s_{3})\rightarrow(s_{1},s_{3},s_{2})$ is an involution on the set of such triples whose fixed points identify with the pairs $(s_{1},s_{2})$ as in the statement of the theorem. This gives the first assertion of the theorem, and a little arithmetic in $\newz [\sqrt{-2}]$ gives the second.
\end{proof}

\begin{theorem}
\label{theorem2.9}
\hspace{.5em} %to force linebreak
\begin{enumerate}
\item[(1)] Every element $n$ of $B(a)$ that lies in $U^{*}$ is the product of a prime and a square.
\item[(2)] The number of elements of $B(a)$ that are $\le x$ and lie in $U^{*}$ is $O\left(x/\log x\right)$.
\end{enumerate}
\end{theorem}

\begin{proof}
The elements of $U^{*}$ are $\equiv 0\pod{2}$, $1\pod{4}$, or $3\pod{8}$, and we use Theorems \ref{theorem2.4}, \ref{theorem2.5} and \ref{theorem2.8} to get (1). (2) is an immediate consequence.
\end{proof}

\begin{remark}
\label{remark1}
The proof of Theorem \ref{theorem2.9} is easier than that of a similar result in Monsky \cite{2}, which makes use of results of Gauss on representations by sums of 3 squares.
\end{remark}

\begin{remark}
\label{remark2}
The set $B(a+a^{4})$ has been more extensively studied. One sees immediately that $a+a^{4}=\sum x^{1+24s}$, where $s$ runs over the generalized pentagonal numbers ${0,1,2,5,7,12,15,\ldots}$. So the elements of $B(a+a^{4})$ are all $\equiv -1\pod{24}$. The mod 2 reduction of a famous identity of Euler tells us that $24k-1$ is in $B(a+a^{4})$ if and only if the number of partitions, $p(k)$, of $k$ is odd. Large-scale computer calculations suggest very strongly that the $k$ for which $p(k)$ is odd have density 1/2, so that $B(a+a^{4})$ has relative density 1/2 in the congruence class $n\equiv -1\pod{24}$. It's tempting to believe that $B(a)$ also has relative density 1/2 in this congruence class. This would be in line with the (modest) computer calculations that have been made; see our final section.
\end{remark}

\section{Enter modular forms. The quintic theta relations}
\label{section3}

In the proofs of section \ref{section2} we expressed $p_{2,0}\left(\frac{1}{a}\right)$, $p_{4,1}\left(\frac{1}{a}\right)$, and $p_{8,3}\left(\frac{1}{a}\right)$ as elements of $\newz/2[a]$, and were able to deduce that $B(a)$ has density 0 in the congruence classes $n\equiv 0\pod{2}$, $n\equiv 1\pod{4}$ and $n\equiv 3\pod{8}$. (Note that $p_{8,7}\left(\frac{1}{a}\right)$ is not in $\newz/2[a]$.  Indeed $p_{8,7}\left(\frac{1}{a}\right)=x^{-1}+\cdots $ and is not even in $\newz/2[[x]]$). In our treatment of larger $l$ we'll use a similar idea, but in most cases we'll have to rely on a deep result on modular forms due to Deligne and Serre. My thanks go to David Rohrlich for telling me about this result.

The following is well-known; for a more general theorem on definite quadratic forms in an even number of variables see Sch\"{o}neberg \cite{4}.

\begin{theorem}
\label{theorem3.1}
$\sum\sum\conste^{2\pi i(m^{2}+n^{2})z}$, the sum extending over all pairs $(m,n)$ with $m$ and $n$ in $\newz$ and $n\equiv\mbox{ some } j \mod l$, converges in $\operatorname{Im}(z)>0$ to a weight 1 modular form for a congruence group.
\end{theorem}

\begin{corollary}
\label{corollary3.2}
Fix $l$. Let $u=\sum a_{s}x^{s}$ be a product of powers of various $[j]$. Then there are integers $c_{0}, c_{1},\ldots$ such that:
\begin{enumerate}
\item[(A)] $\sum c_{n}\conste^{2\pi i nz}$ converges in $\operatorname{Im}(z)>0$ to a modular form of integral weight for a congruence group.
\item[(B)] The mod 2 reduction of $c_{s}$ is $a_{s}$.
\end{enumerate}
\end{corollary}

\begin{proof}
It's enough to show this when $u=[j]$. We take our modular form to be that of Theorem \ref{theorem3.1}. If we write this form as $\sum c_{s}\conste^{2\pi i sz}$, then (A) is satisfied. Furthermore $c_{s}$ is the number of pairs $(m,n)$ with $n\equiv j\pod{l}$ and $m^{2}+n^{2}=s$. $(m,n)\rightarrow (-m,n)$ is an involution on this set of pairs. There is one fixed point if $s$ is the square of some $n\equiv j\pod{l}$, and no fixed point otherwise. It follows that the mod 2 reduction of $c_{s}$ is $a_{s}$.
\end{proof}

Now fix $l$. Recall that $S$ is the subring of $\newz/2[[x]]$ generated over $\newz/2$ by all the $[j]$.

\begin{theorem}
\label{theorem3.3}
If $u=\sum a_{n}x^{n}$ is in $S$, then the set of $n$ for which $a_{n}$ is 1 has density 0.
\end{theorem}

\begin{proof}
We may assume that $u$ is a product of powers of various $[j]$. As we've seen, there are $c_{n}$ in $\newz$, with $c_{n}$ reducing to $a_{n}$ mod 2, such that $\sum c_{n}\conste^{2\pi i nz}$ converges in $\operatorname{Im}(z)>0$ to a modular form of integral weight for a congruence group. A theorem of Serre \cite{5}, based on results of Deligne attaching Galois representations to Hecke eigenforms, shows that the $n$ for which 2 does not divide $c_{n}$ form a set of density 0.
\end{proof}

\begin{corollary}
\label{corollary3.4}
Suppose that $p_{q,j}\left(1/[r]\right)$ is in $S$, or more generally is in $p_{q,j}(S)$. Then $B([r])$ has relative density 0 in the congruence class $j\mod q$.
\end{corollary}

Now $p_{q,j}\left(1/[r]\right)=\left(1/[r]^{q}\right)p_{q,j}\left([r]^{q-1}\right)$. But to show that this quotient lies in $p_{q,j}(S)$ for various choices of $j$ and $q$ seems very difficult. There is however a technique for showing that a quotient of two elements of $S$ lies in $S$ that makes use of certain ``quintic theta relations''.

\begin{lemma}
\label{lemma3.5}
$p_{2,0}\left([2i][2j]\right)=[i+j]^{2}[i-j]^{2}$.
\end{lemma}

\begin{proof}
It suffices to show that the coefficients of $x^{2n}$ on the two sides are equal. On the left one has the mod 2 reduction of the number of pairs $(r,s)$ with $r\equiv 2i\pod{l}$, $s\equiv 2j\pod{l}$ and $r^{2}+s^{2}=2n$. On the right one has the mod 2 reduction of the number of pairs $(t,u)$ with $t\equiv i+j\pod{l}$, $u\equiv i-j\pod{l}$ and $t^{2}+u^{2}=n$. Clearly $(r,s)\rightarrow \left(\frac{r+s}{2},\frac{r-s}{2}\right)$ gives the desired bijection.
\end{proof}

\begin{theorem}
\label{theorem3.6}
$[i]^{4}[2j]+[j]^{4}[2i]+[2i][2j]+[i+j]^{2}[i-j]^{2}=0$.
\end{theorem}

\begin{proof}
$p_{2,0}\left([2i][2j]\right)=p_{2,0}([2i])p_{2,0}([2j])+p_{2,1}([2i])p_{2,1}([2j])=[i]^{4}[j]^{4}+\linebreak \left([i]^{4}+[2i]\right)\left([j]^{4}+[2j]\right)$. Now use Lemma \ref{lemma3.5}.
\end{proof}

Let $x_{1},\ldots ,x_{m}$ (where $l=2m+1$) be indeterminates over $\newz/2$.

\begin{definition}
\label{def3.7}
If $r$ is prime to $l$, $\phi_{r}$ is the homomorphism $\newz/2[x_{1},\ldots , x_{m}]\rightarrow S$ taking $x_{k}$ to $[rk]$.
\end{definition}

Note that each $\phi_{r}$ is onto. We'll use Theorem \ref{theorem3.6} to construct $\frac{m(m-1)}{2}$ elements of $\newz/2[x_{1},\ldots , x_{m}]$ lying in the kernel of each $\phi_{r}$.

\begin{theorem}
\label{theorem3.8}
Suppose that $m\ge i > j \ge 1$. For $1\le k \le m$ define $x_{l-k}$ to be $x_{k}$, so that we have elements $x_{1},\ldots , x_{2m}$ of $\newz/2[x_{1},\ldots , x_{m}]$. Then if we define $R_{i,j}$ to be $x_{i}^{4}x_{2j}+x_{j}^{4}x_{2i}+x_{2i}x_{2j}+x_{i+j}^{2}x_{i-j}^{2}$, each $R_{i,j}$ is in the kernel of each $\phi_{r}$.
\end{theorem}

\begin{proof}
The definition of $x_{m+1},\ldots , x_{2m}$ shows that $\phi_{r}(x_{k})=[rk]$ for $k=1,\ldots , 2m$. The result now follows from Theorem \ref{theorem3.6} on replacing $i$ and $j$ by $ri$ and $rj$ throughout.
\end{proof}

\begin{theorem}
\label{theorem3.9}
Let $u$ and $v$ be elements of $\newz/2[x_{1},\ldots ,x_{m}]$, and $N$ the ideal in this ring generated by the $R_{i,j}$. Suppose that the ideals $(N,v)$ and $(N,u,v)$ are the same. Then the element $\phi_{r}(u)/\phi_{r}(v)$ of the field of fractions of $S$ in fact lies in $S$.
\end{theorem}

\begin{proof}
$u$ is in $(N,v)$. Applying $\phi_{r}$ and using Theorem \ref{theorem3.9} we find that in $S$, $\phi_{r}(u)$ lies in the principal ideal $\phi_{r}(v)$.
\end{proof}

\begin{remark*}
Commutative algebra computer programs such as Macaulay 2 use Gr\"{o}bner bases to decide whether 2 ideals in a polynomial ring are equal. We shall use such a program to show that in many cases of interest the quotient $\phi_{r}(u)/\phi_{r}(v)$ lies in $S$.
\end{remark*}

There is one further simple result that we'll use frequently in the calculations to follow.

\begin{lemma}
\label{lemma3.10}
Suppose that for some $a$ and $b$, $p_{2,0}(a)=b^{4}$. Then:
\begin{enumerate}
\item[(1)] $p_{2,0}\left(\frac{1}{a}\right)=\frac{b^{4}}{a^{2}}$
\item[(2)] $p_{4,0}\left(\frac{1}{a}\right)=\frac{b^{12}}{a^{4}}$
\item[(3)] $p_{8,0}\left(\frac{1}{a}\right)=\frac{b^{8}}{a^{8}}\left(p_{2,0}(ab)\right)^{4}$
\end{enumerate}
\end{lemma}

\begin{proof}
$p_{2,0}\left(\frac{1}{a}\right)=\frac{1}{a^{2}}p_{2,0}(a)=\frac{b^{4}}{a^{2}}$. Then $p_{4,0}\left(\frac{1}{a}\right)=p_{4,0}p_{2,0}\left(\frac{1}{a}\right)=p_{4,0}\left(\frac{b^{4}}{a^{2}}\right)=b^{4}\left(p_{2,0}\left(\frac{1}{a}\right)\right)^{2}=\frac{b^{12}}{a^{4}}$. Furthermore, $p_{8,0}\left(\frac{1}{a}\right)=p_{8,0}p_{4,0}\left(\frac{1}{a}\right)=p_{8,0}\left(\frac{b^{12}}{a^{4}}\right)=\linebreak \frac{b^{8}}{a^{8}}p_{8,0}\left(a^{4}b^{4}\right)$, giving the last result.
\end{proof}

\section{$\bm{l=5}$}
\label{section4}
In this section $l=5$, so that $m=2$. Then the ideal $N$ of Theorem \ref{theorem3.9} is generated by the single element $R_{2,1}=x_{2}^{4}x_{2}+x_{1}^{4}x_{4}+x_{4}x_{2}+x_{3}^{2}x_{1}^{2}=x_{1}^{5}+x_{2}^{5}+x_{1}x_{2}+x_{1}^{2}x_{2}^{2}$. Now let $r=1$ or $2$ and set $a=[r]$, $b=[2r]$. Then $p_{2,0}(a)=b^{4}$, $p_{2,0}(b)=a^{4}$ and we have the quintic relation $a^{5}+b^{5}+ab+a^{2}b^{2}=0$.

We'll use the techniques sketched in the last section to show that $p_{4,1}\left(\frac{1}{a}\right)$, $p_{4,2}\left(\frac{1}{a}\right)$, $p_{8,0}\left(\frac{1}{a}\right)$, $p_{8,3}\left(\frac{1}{a}\right)$, $p_{16,4}\left(\frac{1}{a}\right)$ and $p_{32,12}\left(\frac{1}{a}\right)$ are all in $S$. Corollary \ref{corollary3.4} in conjunction with the description of $U^{*}$ given in the introduction when $l=5$ then tells us that $B(a)$ has relative density 0 in $U^{*}$.

\begin{theorem}
\label{theorem4.1}
$p_{8,0}\left(\frac{1}{a}\right)=b^{16}$.
\end{theorem}

\begin{proof}
By Lemma \ref{lemma3.10}, $p_{8,0}\left(\frac{1}{a}\right)=\frac{b^{8}}{a^{8}}\left(p_{2,0}(ab)\right)^{4}$. Now $p_{2,0}(ab)=p_{2,0}([4r][2r])=[3r]^{2}\cdot [r]^{2}=a^{2}b^{2}$.
\end{proof}

\begin{theorem}
\label{theorem4.2}
$p_{4,2}\left(\frac{1}{a}\right)$, $p_{4,1}\left(\frac{1}{a}\right)$ and $p_{8,3}\left(\frac{1}{a}\right)$ are in $S$.
\end{theorem}

\begin{proof}
We first write these power series as quotients of elements of $S$.
\begin{enumerate}
\item[(1)] $p_{4,2}\left(\frac{1}{a}\right)=p_{2,0}\left(\frac{1}{a}\right)+p_{4,0}\left(\frac{1}{a}\right)=\frac{b^{4}}{a^{2}}+\frac{b^{12}}{a^{4}}=\left(\frac{b^{4}}{a^{4}}\right)\left(a^{2}+b^{8}\right)$.
\item[(2)] $p_{4,1}\left(\frac{1}{a}\right)=\left(\frac{1}{a^{4}}\right)p_{4,1}(a)p_{4,0}(a^{2})=\left(\frac{1}{a^{4}}\right)p_{2,1}(a)\left(p_{2,0}(a)\right)^{2}=\left(\frac{b^{8}}{a^{4}}\right)\left(a+b^{4}\right)$.
\item[(3)] $p_{8,3}\left(\frac{1}{a}\right)=\left(\frac{1}{a^{8}}\right)p_{8,1}(a)p_{8,2}(a^{2})p_{8,0}(a^{4})=\left(\frac{1}{a^{8}}\right)p_{2,1}(a)\left(p_{2,1}(a)\right)^{2}\left(p_{2,0}(a)\right)^{4}=\left(\frac{b^{16}}{a^{8}}\right)\left(a+b^{4}\right)^{3}$.
\end{enumerate}

In view of (1), (2) and (3) it will suffice to show that $\frac{b^{2}}{a^{2}}(a+b^{4})$ and $\frac{b^{8}}{a^{4}}(a+b^{4})$ are each in $S$. This can be done by hand, but in the mechanized spirit of the paper I'll give a computer argument. First let $u=x_{2}^{2}(x_{1}+x_{2}^{4})$ and $v=x_{1}^{2}$. Macaulay 2 tells us that $(N,v)=(N,u,v)$. So by Theorem \ref{theorem3.9}, $\phi_{r}(u)/\phi_{r}(v)$ is in $S$. But $\phi_{r}(u)/\phi_{r}(v)=\frac{b^{2}}{a^{2}}(a+b^{4})$. For the second result we argue similarly taking $u=x_{2}^{8}(x_{1}+x_{2}^{4})$ and $v=x_{1}^{4}$.
\end{proof}

\begin{lemma}
\label{lemma4.3}
$p_{8,4}\left(\frac{1}{a}\right)+\left(p_{2,1}\left(\frac{1}{b}\right)\right)^{4}=a^{4}+b^{16}$.
\end{lemma}

\begin{proof}
$p_{8,4}\left(\frac{1}{a}\right)=p_{4,0}\left(\frac{1}{a}\right)+p_{8,0}\left(\frac{1}{a}\right)=\frac{b^{12}}{a^{4}}+b^{16}$, by Lemma \ref{lemma3.10} and Theorem \ref{theorem4.1}. Furthermore $p_{2,1}\left(\frac{1}{b}\right)=\frac{1}{b}+p_{2,0}\left(\frac{1}{b}\right)=\frac{1}{b}+\frac{a^{4}}{b^{2}}$. So the left hand side in the statement of Lemma \ref{lemma4.3} is $b^{16}+\left(\frac{b^{3}}{a}+\frac{1}{b}+\frac{a^{4}}{b^{2}}\right)^{4}$. But the quintic relation $a^{5}+b^{5}+ab+a^{2}b^{2}=0$ tells us that $\frac{b^{3}}{a}+\frac{1}{b}+\frac{a^{4}}{b^{2}}=\frac{1}{ab^{2}}\left(b^{5}+ab+a^{5}\right)=a$.
\end{proof}

\begin{theorem}
\label{theorem4.4}
$p_{16,4}\left(\frac{1}{a}\right)$ and $p_{32,12}\left(\frac{1}{a}\right)$ are in $S$.
\end{theorem}

\begin{proof}
Applying $p_{16,4}$ to the identity of Lemma \ref{lemma4.3} we find that $p_{16,4}\left(\frac{1}{a}\right)+\left(p_{4,1}\left(\frac{1}{b}\right)\right)^{4}=\left(p_{4,1}(a)\right)^{4}=a^{4}+b^{16}$. But Theorem \ref{theorem4.2} (with $r$ replaced by $2r$) tells us that $p_{4,1}\left(\frac{1}{b}\right)$ is in $S$. Applying $p_{32,12}$ to the identity of Lemma \ref{lemma4.3} we find that $p_{32,12}\left(\frac{1}{a}\right)+\left(p_{8,3}\left(\frac{1}{b}\right)\right)^{4}=\left(p_{8,3}(a)\right)^{4}=0$. And Theorem \ref{theorem4.2} (with $r$ replaced by $2r$) shows that $p_{8,3}\left(\frac{1}{b}\right)$ is in $S$.
\end{proof}

\section{$\bm{l=7}$}
\label{section5}
In this section $l=7$. Then $m=3$ and the ideal $N$ is generated by $x_{1}^{5}+x_{3}^{4}x_{2}+x_{1}x_{2}+x_{2}^{2}x_{3}^{2}$, $x_{2}^{5}+x_{1}^{4}x_{3}+x_{2}x_{3}+x_{3}^{2}x_{1}^{2}$ and $x_{3}^{5}+x_{2}^{4}x_{1}+x_{3}x_{1}+x_{1}^{2}x_{2}^{2}$. Let $r$ be 1, 2 or 3, $a=[r]$, $b=[4r]$, $c=[2r]$. Then $p_{2,0}$ takes $a$, $b$ and $c$ to $b^{4}$, $c^{4}$ and $a^{4}$. Lemma \ref{lemma3.5} shows that $p_{2,0}$ takes $ab$, $bc$ and $ac$ to $a^{2}c^{2}$, $a^{2}b^{2}$ and $b^{2}c^{2}$. We'll prove that $B(a)$ has relative density 0 in $U^{*}$ by showing that each of $p_{4,1}\left(\frac{1}{a}\right)$, $p_{8,0}\left(\frac{1}{a}\right)$, $p_{8,2}\left(\frac{1}{a}\right)$, $p_{8,3}\left(\frac{1}{a}\right)$, $p_{16,4}\left(\frac{1}{a}\right)$, $p_{16,6}\left(\frac{1}{a}\right)$ and $p_{32,12}\left(\frac{1}{a}\right)$ is in $S$.

\begin{remark*}
In this case, as in the case $l=5$, $N$ is the kernel of each $\phi_{r}$. This is not true when $l=9$. Whether it holds for all prime $l$ is an interesting question.
\end{remark*}

\begin{theorem}
\label{theorem5.1}
$p_{8,0}\left(\frac{1}{a}\right)=b^{8}c^{8}$, and $p_{8,2}\left(\frac{1}{a}\right)=\left(a^{2}+b^{8}\right)c^{8}$.
\end{theorem}

\begin{proof}
By Lemma \ref{lemma3.10}, $p_{8,0}\left(\frac{1}{a}\right)=\left(\frac{b^{8}}{a^{8}}\right)\left(p_{2,0}(ab)\right)^{4}=b^{8}c^{8}$. Also, $p_{8,2}\left(\frac{1}{a}\right)=p_{8,2}p_{2,0}\left(\frac{1}{a}\right)=p_{8,2}\left(\frac{b^{4}}{a^{2}}\right)=\left(p_{4,1}\left(\frac{b^{2}}{a}\right)\right)^{2}$.  And $p_{4,1}\left(\frac{b^{2}}{a}\right)=\frac{1}{a^{4}}p_{4,1}(a)p_{4,0}(a^{2}b^{2})=\frac{1}{a^{4}}p_{2,1}(a)\left(p_{2,0}(ab)\right)^{2}=(a+b^{4})\cdot c^{4}$.
\end{proof}

\begin{theorem}
\label{theorem5.2}
$p_{4,1}\left(\frac{1}{a}\right)$, $p_{8,3}\left(\frac{1}{a}\right)$ and $p_{16,6}\left(\frac{1}{a}\right)$ are in $S$.
\end{theorem}

\begin{proof}
Again we first write these power series as quotients of elements in $S$.
\begin{enumerate}
\item[(1)] The proof of Theorem \ref{theorem4.2} shows that $p_{4,1}\left(\frac{1}{a}\right)=\frac{b^{8}}{a^{4}}(a+b^{4})$, and that $p_{8,3}\left(\frac{1}{a}\right)=\frac{b^{16}}{a^{8}}\left(a+b^{4}\right)^{3}$.
\item[(2)] $p_{16,6}\left(\frac{1}{a}\right)=p_{16,6}p_{2,0}\left(\frac{1}{a}\right)=p_{16,6}\left(\frac{b^{4}}{a^{2}}\right)=\left(p_{8,3}\left(\frac{b^{2}}{a}\right)\right)^{2}$. Now $p_{8,3}\left(\frac{b^{2}}{a}\right)=\frac{1}{a^{8}}p_{8,1}(a)p_{8,0}(a^{4})p_{8,2}(a^{2}b^{2})=\frac{1}{a^{8}}p_{2,1}(a)\left(p_{2,0}(a)\right)^{4}\cdot\left(p_{4,1}(ab)\right)^{2}$. Now $p_{4,3}(ab)=0$, and it follows that $p_{4,1}(ab)=p_{2,1}(ab)=ab+p_{2,0}(ab)=ab+a^{2}c^{2}$. So $p_{8,3}\left(\frac{b^{2}}{a}\right)=\left(\frac{b^{16}}{a^{8}}\right)(a+b^{4})(a^{2}b^{2}+a^{4}c^{4})$.
\end{enumerate}

We can now use the technique of the last section to prove the theorem. It suffices to show that $\left(\frac{b^{8}}{a^{4}}\right)(a+b^{4})$ and $\left(\frac{b^{16}}{a^{8}}\right)(a+b^{4})(a^{2}b^{2}+a^{4}c^{4})$ are in $S$. To prove the second result we take $u$ to be $x_{3}^{16}(x_{1}+x_{3}^{4})(x_{1}^{2}x_{3}^{2}+x_{1}^{4}x_{2}^{4})$, and $v$ to be $x_{1}^{8}$. Macaulay 2 verifies that $(N,v)=(N,u,v)$. So $\phi_{r}(u)/\phi_{r}(v)$ is in $S$, as desired. The first result is proved similarly.
\end{proof}

\begin{lemma}
\label{lemma5.3}
$p_{8,4}\left(\frac{1}{a}\right)+\left(p_{4,2}\left(\frac{1}{c}\right)\right)^{2}=u^{4}$ for some $u$ in $S$.
\end{lemma}

\begin{proof}
$p_{8,4}\left(\frac{1}{a}\right)=p_{4,0}\left(\frac{1}{a}\right)+p_{8,0}\left(\frac{1}{a}\right)$. By Lemma \ref{lemma3.10} and Theorem \ref{theorem5.1}, this is $\frac{b^{12}}{a^{4}}+b^{8}c^{8}$. And $p_{4,2}\left(\frac{1}{c}\right)=p_{2,0}\left(\frac{1}{c}\right)+p_{4,0}\left(\frac{1}{c}\right)=\frac{a^{4}}{c^{2}}+\frac{a^{12}}{c^{4}}$. So the left-hand side in the statement of the lemma is the fourth power of $\frac{b^{3}}{a}+b^{2}c^{2}+\frac{a^{2}}{c}+\frac{a^{6}}{c^{2}}$. To show that $\frac{b^{3}}{a}+\frac{a^{2}}{c}+\frac{a^{6}}{c^{2}}$ is in $S$, we write it as a quotient, $\frac{b^{3}c^{2}+a^{3}c+a^{7}}{ac^{2}}$, and use our Macaulay 2 technique.
\end{proof}

\begin{theorem}
\label{theorem5.4}
$p_{16,4}\left(\frac{1}{a}\right)$ and $p_{32,12}\left(\frac{1}{a}\right)$ are in $S$.
\end{theorem}

\begin{proof}
Applying $p_{16,4}$ to the identity of Lemma \ref{lemma5.3} we find that $p_{16,4}\left(\frac{1}{a}\right)+\left(p_{8,2}\left(\frac{1}{c}\right)\right)^{2}=\left(p_{4,1}(u)\right)^{4}$. Now Theorem \ref{theorem5.1} (with $r$ replaced by $2r$) shows that $p_{8,2}\left(\frac{1}{c}\right)$ is in $S$. Since $S$ is stable under $p_{4,1}$, $p_{16,4}\left(\frac{1}{a}\right)$ is in $S$. Similarly, applying $p_{32,12}$ to the identity, we find that $p_{32,12}\left(\frac{1}{a}\right)+\left(p_{16,6}\left(\frac{1}{c}\right)\right)^{2}=\left(p_{8,3}(u)\right)^{4}$. Theorem \ref{theorem5.2} shows that $p_{16,6}\left(\frac{1}{c}\right)$ is in $S$, and we use the fact that $p_{8,3}$ stabilizes $S$.
\end{proof}

\section{$\bm{l=9}$}
\label{section6}

Now $l=9$. Then $m=4$ and $N$ is generated by $x_{1}^{5}+x_{4}^{4}x_{2}+x_{1}x_{2}+x_{3}^{2}x_{4}^{2}$, $x_{2}^{5}+x_{1}^{4}x_{4}+x_{2}x_{4}+x_{3}^{2}x_{1}^{2}$, $x_{4}^{5}+x_{2}^{4}x_{1}+x_{4}x_{1}+x_{3}^{2}x_{2}^{2}$, $x_{1}^{4}x_{3}+x_{3}^{4}x_{2}+x_{2}x_{3}+x_{2}^{2}x_{4}^{2}$, $x_{2}^{4}x_{3}+x_{3}^{4}x_{4}+x_{4}x_{3}+x_{4}^{2}x_{1}^{2}$, and $x_{4}^{4}x_{3}+x_{3}^{4}x_{1}+x_{1}x_{3}+x_{1}^{2}x_{2}^{2}$. Let $r$ be 1, 2 or 4, $a=[r]$, $b=[4r]$, $c=[2r]$ and $d=[3r]=[6r]$. Then $p_{2,0}(d)=d^{4}$, and $p_{2,0}$ takes $a$, $b$ and $c$ to $b^{4}$, $c^{4}$ and $a^{4}$. Lemma \ref{lemma3.5} shows that $p_{2,0}$ takes $ab$, $bc$ and $ac$ to $c^{2}d^{2}$, $a^{2}d^{2}$ and $b^{2}d^{2}$, and that it takes $ad$, $bd$ and $cd$ to $a^{2}c^{2}$, $a^{2}b^{2}$ and $b^{2}c^{2}$. We'll prove that $B(a)$ has relative density 0 in $U^{*}$ by showing that each of $p_{4,2}\left(\frac{1}{a}\right)$, $p_{8,3}\left(\frac{1}{a}\right)$, $p_{8,5}\left(\frac{1}{a}\right)$, $p_{16,4}\left(\frac{1}{a}\right)$, $p_{16,8}\left(\frac{1}{a}\right)$, $p_{32,0}\left(\frac{1}{a}\right)$, $p_{64,16}\left(\frac{1}{a}\right)$ and $p_{128,48}\left(\frac{1}{a}\right)$ is in $S$.

\begin{theorem}
\label{theorem6.1}
$p_{4,2}\left(\frac{1}{a}\right)$, $p_{8,3}\left(\frac{1}{a}\right)$, $p_{8,5}\left(\frac{1}{a}\right)$, $p_{16,4}\left(\frac{1}{a}\right)$ and $p_{16,8}\left(\frac{1}{a}\right)$ are in $S$.
\end{theorem}

\begin{proof}
Again we first write these power series as quotients of elements in $S$.
\begin{enumerate}
\item[(1)] The proof of Theorem \ref{theorem4.2} shows that $p_{4,2}\left(\frac{1}{a}\right)=\left(\frac{b^{4}}{a^{4}}\right)\left(a^{2}+b^{8}\right)$ while $p_{8,3}\left(\frac{1}{a}\right)=\left(\frac{b^{16}}{a^{8}}\right)\left(a+b^{4}\right)^{3}$.
\item[(2)] $p_{8,5}\left(\frac{1}{a}\right)=\left(\frac{1}{a^{8}}\right)p_{8,1}(a)p_{8,0}(a^{2})p_{8,4}(a^{4})=\left(\frac{1}{a^{8}}\right)p_{2,1}(a)\left(p_{4,0}(a)\right)^{2}\left(p_{2,1}(a)\right)^{4}=\frac{b^{8}}{a^{8}}(a+b^{4})^{5}$.
\item[(3)] $p_{16,4}\left(\frac{1}{a}\right)=p_{16,4}p_{4,0}\left(\frac{1}{a}\right)=\left(p_{4,1}\left(\frac{b^{3}}{a}\right)\right)^{4}$. And $p_{4,1}\left(\frac{b^{3}}{a}\right)=\left(\frac{1}{a^{4}}\right)p_{4,1}(ab) p_{4,0}(a^{2}b^{2})\linebreak =\left(\frac{1}{a^{4}}\right)p_{2,1}(ab)\left(p_{2,0}(ab)\right)^{2}=\frac{c^{4}d^{4}}{a^{4}}(ab+c^{2}d^{2})$.
\item[(4)] $p_{8,0}\left(\frac{1}{a}\right)=\left(\frac{b^{8}}{a^{8}}\right)\left(p_{2,0}(ab)\right)^{4}=\left(\frac{bcd}{a}\right)^{8}$. If follows that $p_{16,8}\left(\frac{1}{a}\right)=p_{16,8}p_{8,0}\left(\frac{1}{a}\right)=\left(p_{2,1}\left(\frac{bcd}{a}\right)\right)^{8}$.  Now $p_{2,1}\left(\frac{bcd}{a}\right)=\left(\frac{1}{a^{2}}\right)p_{2,1}\left((ab)(cd)\right)=\left(\frac{1}{a^{2}}\right)
\linebreak \left((c^{2}d^{2})(cd+b^{2}c^{2})+ (ab+c^{2}d^{2})(b^{2}c^{2})\right)=\left(\frac{1}{a^{2}}\right)(ab^{3}c^{2}+c^{3}d^{3})$.
\end{enumerate}

We conclude with our by now standard computer procedure. For example to show that $\left(\frac{1}{a^{2}}\right)(ab^{3}c^{2}+c^{3}d^{3})$ is in $S$ we set $u=x_{1}x_{4}^{3}x_{2}^{2}+x_{2}^{3}x_{3}^{3}$, $v=x_{1}^{2}$ and use Macaulay 2 to verify that $(N,v)=(N,u,v)$.
\end{proof}

\begin{lemma}
\label{lemma6.2}
$p_{16,0}\left(\frac{1}{a}\right)$ is the sixteenth power of $\frac{d(ab^{2}+bc^{2}+ca^{2})}{a}$.
\end{lemma}

\begin{proof}
Arguing as in the above calculation of $p_{16,8}\left(\frac{1}{a}\right)$ we find that $p_{16,0}\left(\frac{1}{a}\right)$ is the eighth power of $p_{2,0}\left(\frac{bcd}{a}\right)=\frac{abcd}{a^{2}}+\frac{ab^{3}c^{2}+c^{3}d^{3}}{a^{2}}$. So it suffices to show that $(abcd+ab^{3}c^{2}+c^{3}d^{3})+d^{2}(a^{2}b^{4}+b^{2}c^{4}+c^{2}a^{4})=0$. To do this, set $u=(x_{1}x_{4}x_{2}x_{3}+x_{1}x_{4}^{3}x_{2}^{2}+x_{2}^{3}x_{3}^{3})+x_{3}^{2}(x_{1}^{2}x_{4}^{4}+x_{4}^{2}x_{2}^{4}+x_{2}^{2}x_{1}^{4})$. Macaulay 2 shows that $(N,u)=N$. So $u$ is in $N$ and applying $\phi_{r}$ gives the result.
\end{proof}

\begin{lemma}
\label{lemma6.3}
$p_{16,0}\left(\frac{1}{a}\right)+\left(p_{8,4}\left(\frac{1}{b}\right)\right)^{4}=u^{16}$ for some $u$ in $S$.
\end{lemma}

\begin{proof}
$p_{8,4}\left(\frac{1}{b}\right)=p_{8,0}\left(\frac{1}{b}\right)+p_{4,0}\left(\frac{1}{b}\right)$. Using Lemma \ref{lemma3.10} we find that this is $\left(\frac{acd}{b}\right)^{8}+\left(\frac{c^{3}}{b}\right)^{4}$. So the left-hand side in the statement of the lemma is the sixteenth power of $u=\frac{d(ab^{2}+bc^{2}+ca^{2})}{a}+\frac{a^{2}c^{2}d^{2}}{b^{2}}+\frac{c^{3}}{b}$. It remains to show that this $u$ is in $S$. This is established using Macaulay 2 in the usual way.
\end{proof}

\begin{theorem}
\label{theorem6.4}
$p_{32,0}\left(\frac{1}{a}\right)$ and $p_{64,16}\left(\frac{1}{a}\right)$ are in $S$.
\end{theorem}

\begin{proof}
Applying $p_{32,0}$ to the identity of Lemma \ref{lemma6.3} we find that $p_{32,0}\left(\frac{1}{a}\right)=\left(p_{2,0}(u)\right)^{16}$ with $u$ in $S$. Applying $p_{64,16}$ to the identity we find that $p_{64,16}\left(\frac{1}{a}\right)+\left(p_{16,4}\left(\frac{1}{b}\right)\right)^{4}=\left(p_{4,1}(u)\right)^{16}$. But Theorem \ref{theorem6.1} shows that $p_{16,4}\left(\frac{1}{b}\right)$ is in $S$.
\end{proof}

\begin{theorem}
\label{theorem6.5}
$p_{32,12}\left(\frac{1}{a}\right)$ and $p_{128,48}\left(\frac{1}{a}\right)$ are in $S$.
\end{theorem}

\begin{proof}
We show how the second result follows from the first. Applying $p_{128,48}$ to the identity of Lemma \ref{lemma6.3} we find that $p_{128,48}\left(\frac{1}{a}\right)+\left(p_{32,12}\left(\frac{1}{b}\right)\right)^{4}=\left(p_{8,3}(u)\right)^{16}$. Since $p_{32,12}\left(\frac{1}{b}\right)$ is in $S$ and $p_{8,3}$ stabilizes $S$ we get the second result. To prove the first result we once again express our element as a quotient of two elements of $S$. $p_{32,12}\left(\frac{1}{a}\right)=p_{32,12}p_{4,0}\left(\frac{1}{a}\right)=p_{32,12}\left(\frac{b^{12}}{a^{4}}\right)=\left(p_{8,3}\left(\frac{b^{3}}{a}\right)\right)^{4}$. So it's enough to show that $p_{8,3}\left(\frac{b^{3}}{a}\right)$ is in $S$. Now $p_{8,3}\left(\frac{b^{3}}{a}\right)=\left(\frac{1}{a^{8}}\right)p_{8,3}\left((a^{2}b^{2})(ab)(a^{4})\right)=\left(\frac{1}{a^{8}}\right)p_{8,2}(a^{2}b^{2})\left(p_{8,1}(ab)p_{8,0}(a^{4})+p_{8,5}(ab)p_{8,4}(a^{4})\right)$. Now modulo $a^{8}$, $p_{8,1}(ab)=p_{8,0}(a)p_{8,1}(b)=c^{16}(b+c^{4})$. Also $p_{4,1}(ab)=p_{2,1}(ab)=ab+c^{2}d^{2}$. So modulo $a^{8}$, $p_{8,5}(ab)=p_{4,1}(ab)+c^{16}(b+c^{4})=ab+c^{2}d^{2}+c^{16}(b+c^{4})$. We conclude that $p_{8,3}\left(\frac{b^{3}}{a}\right)$ is the sum of an element of $S$ and $\frac{1}{a^{8}}(a^{2}b^{2}+c^{4}d^{4})\linebreak  \left(c^{16}(b+c^{4})b^{16}+(ab+c^{2}d^{2}+c^{16}(b+c^{4}))(a^{4}+b^{16})\right)$. A Macaulay 2 calculation shows that this last element is in $S$.
\end{proof}

\section{$\bm{l=11}$, $\bm{13}$ and $\bm{15}$}
\label{section7}

We state the results for these $l$ with very brief indications of proofs.

\begin{lemma}
\label{lemma7.1}
Let $a=[r]$ with $r$ prime to $l$. Then $p_{8,k}\left(\frac{1}{a}\right)$, $p_{16,2k}\left(\frac{1}{a}\right)$, $p_{32,4k}\left(\frac{1}{a}\right)$ and $p_{64,8k}\left(\frac{1}{a}\right)$ are all quotients of elements of $S$ by powers of $a$.
\end{lemma}

\begin{proof}
$p_{8,k}\left(\frac{1}{a}\right)=\frac{1}{a^{8}}p_{8,k}(a^{7})$, and we use Lemma \ref{lemma2.2}. For the remaining results we may assume that $r=4s$. Let $b=[2s]$, $c=[s]$, $e=[3s]$ so that $p_{2,0}(a)=b^{4}$, $p_{2,0}(ab)=c^{2}e^{2}$. Then $p_{64,8k}\left(\frac{1}{a}\right)=p_{64,8k}p_{8,0}\left(\frac{1}{a}\right)$. By Lemma \ref{lemma3.10} this is the eighth power of $p_{8,k}\left(\frac{bce}{a}\right)$, and we use the fact that $p_{8,k}(a^{7}bce)$ is in $S$. $p_{16,2k}\left(\frac{1}{a}\right)$ and $p_{32,4k}\left(\frac{1}{a}\right)$ are treated similarly.
\end{proof}

\begin{theorem}
\label{theorem7.2}
Let $a=[r]$ with $r$ prime to $l$.
\begin{enumerate}
\item[(1)] When $l=11$, $p_{8,1}$, $p_{8,3}$, $p_{8,6}$, $p_{16,4}$, $p_{16,8}$, $p_{16,10}$, $p_{32,0}$, $p_{32,12}$ and $p_{64,16}$ all take $\frac{1}{a}$ to an element of $S$.
\item[(2)] When $l=13$, $p_{8,2}$, $p_{8,3}$, $p_{8,5}$, $p_{16,4}$, $p_{16,8}$, $p_{16,14}$, $p_{32,0}$, $p_{32,12}$ and $p_{64,16}$ all take $\frac{1}{a}$ to an element of $S$.
\item[(3)] When $l=15$, $p_{8,1}$, $p_{8,2}$, $p_{8,3}$, $p_{16,4}$, $p_{16,6}$, $p_{16,8}$, $p_{32,0}$, $p_{32,12}$ and $p_{64,16}$ all take $\frac{1}{a}$ to an element of $S$.
\end{enumerate}
\end{theorem}

\begin{proofidea*}
By Lemma \ref{lemma7.1} each $p_{q,j}\left(\frac{1}{a}\right)$ is the quotient of an element of $S$ by a power of $a$. It's clear that one can write down such a representation explicitly. In each case the Macaulay 2 argument using the ideal $N$ of quintic relations shows that $p_{q,j}\left(\frac{1}{a}\right)$ is in fact in $S$.
\end{proofidea*}

\begin{corollary}
\label{corollary7.3}
Suppose $l=11$, $13$, or $15$. Then in each of the mod 128 congruence classes constituting $U^{*}$, with the possible exception of the congruence class $n\equiv 48\pod{128}$, $B(a)$ has relative density 0
\end{corollary}

\begin{proof}
This follows from Theorem \ref{theorem7.2}, Corollary \ref{corollary3.4} and the explicit description of $U^{*}$ as a union of congruence classes.
\end{proof}

I'll now show that when $l=11$ each $B(a)$ in fact has relative density 0 in the congruence class 48 mod 128.

\begin{lemma}
\label{lemma7.4}
When $l=11$, $p_{8,0}\left(\frac{1}{a}\right)+\left(p_{8,4}\left(\frac{1}{b}\right)\right)^{4}=u^{8}$ for some $u$ in $S$.
\end{lemma}

\begin{proofidea*}
As we noted in the proof of Lemma \ref{lemma7.1}, $p_{8,0}\left(\frac{1}{a}\right)=\left(\frac{bce}{a}\right)^{8}$. Furthermore $p_{8,4}\left(\frac{1}{b}\right)=p_{8,4}p_{2,0}\left(\frac{1}{b}\right)=\left(\frac{1}{b^{8}}\right)p_{8,4}(b^{6}c^{4})$. This is  the quotient of a square in $S$ by $b^{8}$. It follows that the left-hand side in the statement of Lemma \ref{lemma7.4} is the eighth power of $\frac{v}{ab^{4}}$ for some $v$ in $S$. Our usual Macaulay 2 technique shows that $\frac{v}{ab^{4}}$ is in fact in $S$.
\end{proofidea*}

\begin{theorem}
\label{theorem7.5}
When $l=11$, $p_{128,48}\left(\frac{1}{a}\right)$ is in $p_{128,48}(S)$. In fact it's the eighth power of an element of $p_{16,6}(S)$. Corollary \ref{corollary3.4} then shows that $B(a)$ has relative density 0 in the congruence class 48 mod 128, and consequently in $U^{*}$.
\end{theorem}

\begin{proof}
Applying $p_{128,48}$ to the identity of Lemma \ref{lemma7.4} we find that $p_{128,48}\left(\frac{1}{a}\right)+\left(p_{32,12}\left(\frac{1}{b}\right)\right)^{4}=\left(p_{16,6}(u)\right)^{8}$. Now $p_{32,12}\left(\frac{1}{b}\right)=p_{32,12}p_{4,0}\left(\frac{1}{b}\right)=p_{32,12}\left(\frac{c^{12}}{b^{4}}\right)$, which is the square of $p_{16,6}\left(\frac{c^{6}}{b^{2}}\right)$.  So $p_{128,48}\left(\frac{1}{a}\right)$ is the eighth power of $p_{16,6}\left(\frac{c^{6}}{b^{2}}\right)+p_{16,6}(u)$, and it will suffice to show that $p_{16,6}\left(\frac{c^{6}}{b^{2}}\right)$ is in $S$. In fact, $p_{8,3}\left(\frac{c^{3}}{b}\right)$ is in $S$; the Macaulay 2 calculations going into the proof of Theorem \ref{theorem7.2} show this.
\end{proof}

\begin{remarks*}
We've established various zero-density results when $l\le 15$. If we take $l>15$, computer trouble arises. Suppose for example we restrict ourselves to congruence classes to the modulus 8 that lie in $U^{*}$. Then necessarily $l\le 21$ or $l=25$. When $l=17$, the classes $n\equiv 5\pod{8}$ and $n\equiv 6\pod{8}$ are in $U^{*}$. But the ideal $N$ in $\newz/2[x_{1},\ldots ,x_{8}]$ has 28 generators, and attempts, using Macaulay 2, to show that $p_{8,5}\left(\frac{1}{a}\right)$ (or $p_{8,6}\left(\frac{1}{a}\right)$) is in $S$ cause a computer crash. Indeed the computer seemed at its limit in handling the congruence class $n\equiv 16\pod{64}$ when $l=15$; it was an all-day calculation.

For $l=11$ I don't know whether Theorem \ref{theorem7.5} can be strengthened to show that $p_{128,48}\left(\frac{1}{a}\right)$ is in $S$. When $l=13$ or $15$ it's possible that, as in the case $l=11$, $p_{128,48}\left(\frac{1}{a}\right)$ is the eighth power of an element of $p_{16,6}(S)$. But there's no analogue of Lemma \ref{lemma7.4} that could be used to prove this.
\end{remarks*}

\section{The basic classes --- a little computer evidence}

Fix $l$ together with $r$ prime to $l$ and a basic congruence class $C$. All the elements of $B([r])$ are $\ge -r^{2}$ and are congruent to $-r^{2}$ mod $l$. There is some evidence that $B([r])$ has density $\frac{1}{2l}$ in $C$, so that ``half the elements of $C$ that are $\ge -r^{2}$ and are congruent to $-r^{2}$ mod $l$ lie in $B([r])$.''

Suppose for example that $l\le 9$ and we are looking at the basic classes to the modulus 8. These are:

\centerline{
\begin{tabular}{lll}
(1) & $l=3$ & $n\equiv 7\pod{8}$\\
(2) & $l=5$ & $n\equiv 7\pod{8}$\\
(3) & $l=7$ & $n\equiv 7\pod{8}$\\
(4) & $l=9$ & $n\equiv 1\mbox{ or }7\pod{8}$\\
\end{tabular}
}

Consider the first $2^{17}=131,072$ elements of $C$ that are $\ge -r^{2}$ and congruent to $-r^{2}$ mod $l$. The number of these lying in $B([r])$ has been calculated by O'Bryant \cite{3}. Here are his results.
\vspace{2ex}

\centerline{
\begin{tabular}{lllllllll}
(1) & $l=3$ & $n\equiv 7\pod{8}$, & $r=1$ & $65,411$\\
(2) & $l=5$ & $n\equiv 7\pod{8}$, & $r=1$ & $65,397$  & $r=2$ & $65,713$\\
(3) & $l=7$ & $n\equiv 7\pod{8}$, & $r=1$ & $65,185$ & $r=2$ & $65,474$  & $r=3$ & $65,622$\\
(4) & $l=9$ & $n\equiv 1\pod{8}$, & $r=1$ & $65,495$  & $r=2$ & $65,666$  & $r=4$ & $65,367$\\
&& $n\equiv 7\pod{8}$ , & $r=1$ & $65,877$  & $r=2$ & $65,579$  & $r=4$ & $65,813$\\
\end{tabular}
}

We may also consider the basic congruence class $n\equiv 14\pod{16}$ when $l=7$. Now if we consider the first 65,536 elements of the class that are $\equiv -r^{2}$ mod 7 and $\ge -r^{2}$, the number in $B([r])$ is 32,673 when $r=1$. It is 32,716 when $r=2$ and 32,981 when $r=3$. All this suggests the following:

\begin{speculation*}
Suppose that $\rho >\frac{1}{2}$. Consider a basic class $C$ and the first $X$ elements in the class that are $\ge -r^{2}$ and congruent to $-r^{2}$ mod $l$. Of these elements, the number in $B([r])$ is $\frac{X}{2}+ O(X^{\rho})$.
\end{speculation*}

We might go even further, speculating that this is true not only for the basic classes, but for any congruence class contained in a a basic class.

It would be interesting to test these speculations further experimentally. But some caution is in order. Suppose for example that $l=9$. Then the congruence class $n\equiv 2\pod{4}$ is contained in $U^{*}$, and as we've seen, $B([1])$, $B([2])$ and $B([4])$ all have relative density 0 in this class. Consider now the first $2^{18}=262,144$ elements of this class that are $\ge -r^{2}$ and congruent to $-r^{2}$ mod 9. The number of these elements that lie in $B([r])$ is 102,284 when $r=1$, and 110,034 when $r=2$. This is in good accord with our zero-density result. But when $r=4$ more than half of the elements are in $B([r])$! (The number is 137,657.) So we are advised not to place too much predictive power in such computer counts unless the range over which we're counting is considerably extended.

%%%%%%%%
%%%%%%%%%%%%%%%%

\label{}

% \appendix

% \section{}
% \label{}

\end{document}